\documentclass[sn-mathphys-num]{sn-jnl}

\UseRawInputEncoding
\usepackage[numbers]{natbib}
\usepackage{lipsum}
\usepackage{amsfonts}
\usepackage{epstopdf}
\usepackage{graphicx}
\usepackage{subfigure}
\usepackage{xfrac}
\usepackage{amsmath}
\usepackage{booktabs}
\numberwithin{equation}{section}
\usepackage{amsthm}
\newtheorem{theorem}{Theorem}[section]
\newtheorem{definition}{Definition}[section]
\newtheorem{remark}{Remark}[section]
\usepackage{algorithm}
\usepackage{algorithmic}
\floatname{algorithm}{Method}
\usepackage{multirow}
\usepackage{caption}
\usepackage{tabularx}

\usepackage{hyperref} 
\hypersetup{
	hidelinks,
	colorlinks=true,
	linkcolor=blue,
	citecolor=blue,
	urlcolor = black
}
\renewcommand{\proofname}{\textbf{Proof}}

\ifpdf
\DeclareGraphicsExtensions{.eps,.pdf,.png,.jpg}
\else
\DeclareGraphicsExtensions{.eps}
\fi

\begin{document}
	
\title[mode=title]{On maximum residual block Kaczmarz method for solving large consistent linear systems}

\author[1]{ \sur{Wen-Ning Sun}}\email{222122189@st.usst.edu.cn}

\author*[1]{ \sur{Mei Qin}}\email{qin5670830@163.com}

\affil[1]{\orgdiv{College of Science}, \orgname{University of Shanghai for Science and Technology}, \city{Shanghai}, \postcode{200093},  \country{P.R. China}}



\abstract{For solving large consistent linear systems by iteration methods, inspired by the maximum residual Kaczmarz method and the randomized block Kaczmarz method, we propose the maximum residual block Kaczmarz method, which is designed to preferentially eliminate the largest block in the residual vector $r_{k}$ at each iteration. At the same time, in order to further improve the convergence rate, we construct the maximum residual average block Kaczmarz method to avoid the calculation of pseudo-inverse in block iteration, which completes the iteration by projecting the iteration vector $x_{k}$ to each row of the constrained subset of $A$ and applying different extrapolation step sizes to average them. We prove the convergence of these two methods and give the upper bounds on their convergence rates, respectively. Numerical experiments validate our theory and show that our proposed methods are superior to some other block Kaczmarz methods.}




\maketitle

\section{Introduction}\label{sec1}

For large-scale linear systems of the following form
\begin{equation}\label{eq:1.1}
	Ax=b,\quad\text{with}\quad A\in\mathbb{R}^{m\times n}\quad\text{and}\quad b\in\mathbb{R}^m
\end{equation}
i.e., with $A$ being an m-by-n real matrix, $b$ being an m-dimensional real
vector, and $x$ being the n-dimensional unknown vector, researchers often employ iterative methods to solve them. One such method is the Kaczmarz method \cite{ref1}, which is known for its simplicity and effectiveness. The Kaczmarz method (also known as ``ART'' in \cite{ref2}), as a row iteration method, completes each iteration by projecting the current point onto the hyperplane formed by the selected row. To be more specific, let $A^{(i)}$ represents the $i$th row of $A$, and $b^{(i)}$ represents the $i$th entry of $b$. Given an initial approximation value $ x_0$, the Kaczmarz method can be expressed as follows:$$x_{k+1}=x_{k}+\frac{(b^{(i_{k})}-A^{(i_{k})}x_{k})}{\parallel A^{(i_{k})}\parallel_{2}^{2}}(A^{(i_{k})})^{T}$$ where ${{(\cdot)^{T}}}$ and $\|\cdot\|_{2}$ denote the transpose and Euclidean norm of a vector or matrix, respectively, and the target row index $i_{k} = mod (k, m) +1$.

The theory of the Kaczmarz method has significant development since its inception. Initially, Kaczmarz \cite{ref1} demonstrated the convergence of this method when the coefficient matrix $A$ is a non-singular square matrix. Subsequently, Gal\'{a}ntai \cite{ref18} provided an upper bound on its convergence rate, while Knight conducted error analysis on it under limited precision operation in \cite{ref19}. For further research literature, please refer to \cite{ref39,ref40}. Due to its simplicity and efficiency, the Kaczmarz method finds wide applications in various large-scale computing fields, including computed tomography \cite{ref20,ref21,ref22,ref23,ref24}, image reconstruction \cite{ref2,ref25,ref26}, signal processing \cite{ref26,ref27}, distributed computing \cite{ref28,ref29}, etc.

The classic Kaczmarz method iteratively updates the solution vector by selecting the working row of the coefficient matrix in sequence and projecting it orthogonally to the hyperplane where the row is located. However, when the scale of the coefficient matrix is very large, the convergence rate can be significantly slow. In 2009, Strohmer and Vershynin \cite{ref3} introduced the Randomized Kaczmarz (RK) method with expected exponential rate of convergence for solving overdetermined consistent linear systems, reigniting interest in Kaczmarz methods.They improved the row selection strategy of the Kaczmarz method, which can accelerate the convergence speed of the kaczmarz method significantly. Specifically, in the $k$th iteration of the RK method, the target row is selected according to $\Pr(\mathrm{row}=i_{k})=\sfrac{\|A^{(i_{k})}\|_{2}^{2}}{\|A\|_{F}^{2}}$, and then the current iteration solution $x_{k}$ is projected onto the hyperplane $\{x| A^{(i_{k})}x=b^{(i_{k})}\}$.

The RK method has obvious flaws in the criterion for selecting rows: if the Euclidean norm of each row of A is the same, then the probability of each row being selected degenerates to 1/m, which means uniform sampling. Additionally, if the Euclidean norm of a certain row in A is much smaller than other rows, that row will hardly be selected, thereby greatly slowing down the convergence rate of the RK method. Therefore, Bai and Wu \cite{ref4} proposed a greedy randomized Kaczmarz(GRK) method, which introduced a greedy probability criterion to obtain the larger component of the module of the residual vector in each iteration, so that it would be eliminated first in the iteration process, and thus accelerate the convergence rate. They proved the convergence of the GRK method for a consistent linear system and the expected error convergence rate. It is worth noting that Ansorge \cite{ref37} proposed a maximal residual Kaczmarz(MRK) method, and Popa analyzed it in \cite{ref38}.  Similar to the work row obtained by the GRK method, this method selects the target working row index $i_{k}$ so that the $i_{k}$-th component of the residual has relatively the largest absolute value compared to other components, i.e. $i_{k}=\arg\max\limits_{1\le i\le m}|b^{(i)}-A^{(i)}x_{k}| $. Based on the GRK method, Bai and Wu \cite{ref5} constructed the relaxed greedy randomized Kaczmarz(RGRK) method by introducing a relaxation factor, and showed that the RGRK method was more effective than the GRK method when the relaxation factor was selected appropriately. In \cite{ref6}, Zhang developed a new greedy Kaczmarz method and proved the convergence of the method. For more research on the randomized Kaczmarz method, refer to \cite{ref12,ref13}.

For iterative solutions of large linear equations, in order to accelerate the convergence rate of Kaczmarz method, it is a natural idea to use block iteration instead of single row iteration, so block Kaczmarz method emerges as the times require. Bai proved the convergence of the block kaczmarz method in \cite{ref39}. Needell et al. pointed out in \cite{ref7} that matrix has good row paving, introduced a block strategy that depends on matrix eigenvalues, and proposed the first block Kaczmarz method with (expected) linear convergence to solve the overdetermined least squares problems, which projected the current iterative solution vector onto the solution space of constrained subsets at each iteration step. To be specific, if the subset $\mathcal{J}_{k}$ is selected at the $k$-th iteration, the iteration formula for $x_{k}$ can be expressed as : $x_{k+1}=x_{k}+A_{\mathcal{J}_{k}}^{\dagger}(b_{\mathcal{J}_{k}}-A_{\mathcal{J}_{k}}x_{k})$ with $A_{\mathcal{J}_{k}}=A({\mathcal{J}}_{k},:)$, $b_{\mathcal{J}_{k}}=b({\mathcal{J}}_{k})$. Needell et al. then proposed a randomized block kaczmarz(RBK) method for solving the least squares  in \cite{ref8}. Necoara \cite{ref9} presented a group of randomized average block Kaczmarz(RABK) methods that involve a subset of the constraints and extrapolated step sizes. As a natural follow-up to the RBK method, Liu and Gu \cite{ref14} proposed the greedy randomized block kaczmarz(GRBK) method for solving consistent linear systems. Inspired by the GRK and RABK methods, Miao and Wu \cite{ref10}  proposed a greedy randomized average block Kaczmarz (GRABK) method to avoid the expensive cost of the GRBK method to solve the pseudo-inverse of the selected submatrix in iteration. Niu and Zheng \cite{ref11} proposed a greedy block Kaczmarz(GBK) method for solving large-scale consistent linear systems, proved the convergence of the method, and showed that the method can be more efficient than the GRK method if parameter $\eta$ is chosen appropriately. For more research on the block Kaczmarz method, refer to \cite{ref15,ref16,ref17,ref30,ref31,ref32,ref33,ref34}.

In this paper, we construct the maximum residual block Kaczmarz (denoted as MRBK) method, which is designed to preferentially eliminate the largest block in the residual vector $r_{k}$ at each iteration. At the same time, in order to avoid the pseudo-inverse calculation of the MRBK method during iteration, we further develop the maximum residual average block Kaczmarz (denoted as MRABK) method, which completes the iteration by projecting the iteration vector $x_{k}$ to each row of  $A_{\mathcal{V}_{i_{k}}}$ and applying different extrapolation step sizes to average them. We prove the convergence of these two methods under a consistent linear system and give the upper bounds on their convergence rates, respectively. Numerical experiments show that MRABK method is superior to MRBK method, and both methods are more efficient than the GRK, MRK, RBK, GRBK and GBK methods.

The organization of this paper is as follows. In Section \ref{sec:2} we introduce the maximum residual block Kaczmarz method and establish its convergence theory. In Section \ref{sec:3} we introduce the maximum residual average block Kaczmarz method and establish its convergence theory. The effectiveness of our proposed methods are verified by numerical experiments in Section \ref{sec:4}. Finally, in Section \ref{sec:5}, we end the paper with conclusions.

\textbf{Notation and some basic assumptions:} For a matrix $A\in\mathbb{R}^{m\times n} $, $\|A\|_{2}$, $\|A\|_{F}$ and $A^{\dagger}$ signify the Euclidean norm, Frobenius norm, and Moore-Penrose pseudoinverse of matrix $A$, respectively. We define $A$ as a standardized matrix if the Euclidean norm of each row of $A$ is equal to 1, i.e. $\|A^{(i)}\|_{2}=1,i=1,2,\ldots,m $. Similarly, for a given vector $u$, $\|u\|_{2}$ also represents its Euclidean norm.   The notation $\sigma_{min}(A)$ and $\sigma_{max}(A)$ are employed to express the smallest nonzero and largest singular values of matrix $A$, respectively. Additionally, let us define the set $[m]$ as $\{1, 2, ...  , m\}$, where $m$ is an arbitrary positive integer. We consider the collection $\mathcal{V}=\{\mathcal{V}_1,\mathcal{V}_2,\ldots,\mathcal{V}_t\}$ as a partition of $[m]$ if the index sets $\mathcal{V}_i$, where $i = 1, 2, \ldots , t$, satisfy the conditions $\mathcal{V}_i\cap\mathcal{V}_j = \emptyset$ for $i\neq j $, and $\cup_{i=1}^{t}\mathcal{V}_i=[m]$. Furthermore, given an row index set $\mathcal{V}_i$, we use $A_{\mathcal{V}_{i}}$ to denote the row submatrix of matrix $A$ indexed by $\mathcal{V}_{i}$, use $b_{\mathcal{V}_{i}}$ to denote the subvector of vector $b$. We use $I$ to represent the identity matrix of appropriate size. Define the randomized partition of $[m]$ as $\mathcal{V}=\{\mathcal{V}_1,\mathcal{V}_2,\ldots,\mathcal{V}_t\}$ with
\begin{equation}\label{eq:1.2}
  \mathcal{V}_{i}=\left\{\pi(k):k=\left\lfloor(i-1)m/t\right\rfloor+1,\left\lfloor(i-1)m/t\right\rfloor+2,\ldots,\left\lfloor im/t\right\rfloor\right\}
\end{equation}
where $i = 1, 2, \ldots , t$, we assume that the row partition anywhere else in this paper is as shown in \eqref{eq:1.2}.

\section{Maximum residual block Kaczmarz method}
\label{sec:2}

In this section, inspired by the idea of the MRK \cite{ref37}, RBK \cite{ref8} and GRBK \cite{ref14} methods, we are going to construct the maximum residual block Kaczmarz (MRBK) method and analyze its convergence property.

There are typically two approaches to accelerate the Kaczmarz method: the first approach focuses on selecting working rows more efficiently to achieve faster convergence in each iteration, while the second approach aims to utilize row block iteration instead of single-row iteration for acceleration. Building upon these approaches, we propose the MRBK method as follow as Method \ref{alg:1}. Firstly, we partition the rows of matrix $A$ to obtain the row block division $\mathcal{V}$ of $A$, i.e. $\{A_{\mathcal{V}_{1}},A_{\mathcal{V}_{2}},\ldots,A_{\mathcal{V}_{t}}\}$. Next, we denote $r_{k}^{(i)}$ as the $i$th block component corresponding to the residual vector $r_{k}$, then $r_{k}^{(i)}=b_{\mathcal{V}_{i}}-A_{\mathcal{V}_{i}}x_{k}$, where $i=1, 2,\ldots,t$. We select the working row block $\mathcal{V}_{i_{k}}$ according to $i_{k}=\arg\max\limits_{1\le i\le t}\|b_{\mathcal{V}_{i}}-A_{\mathcal{V}_{i}}x_{k}\|_2^2 $, ensuring that the largest residual is eliminated first in each iteration within the block. This significantly improves the convergence rate.

Next, we try to analyze the differences and improvements of the MRBK method compared to other three methods:

1. The MRBK method vs. the MRK method: The MRBK method accelerates the MRK method naturally by utilizing row block iterations instead of single row iterations.

2. The MRBK method vs. (the RBK and GRBK methods): In comparison to the RBK method, the GRBK method improves the RBK method by introducing a new greedy probability criterion to randomly select the index of the row blocks, which ensures that row blocks with large residual values are prioritized for elimination, thus accelerating the RBK method. However, the GRBK method requires constructing the index set of row blocks and then selecting them based on probability in each iteration. In contrast, our proposed MRBK method selects the row blocks with the largest residual directly, enhancing the iteration efficiency. In fact, along with the idea of ``greedy", our method can be called ``extremely greedy" when it comes to selecting the index of the row blocks.

\begin{algorithm}
	\caption{the MRBK Method for solving the linear system}
	\label{alg:1}
	\begin{algorithmic}
		\STATE	\textbf{Input:} $A$, $b$, $\ell$, ${x}_{0}$.
		\STATE  \textbf{Output:} ${x}_{\ell}$
		
		\STATE	1: Let $\{\mathcal{V}_1,\mathcal{V}_2,\ldots,\mathcal{V}_t\}$ be a partition of $\mathrm[m]$
		
		\STATE	2: \textbf{for} $k=0,1,\ldots,\ell-1$ \textbf{do}
		\STATE	3: Select $i_{k}=\arg\max\limits_{1\le i\le t}\|b_{\mathcal{V}_{i}}-A_{\mathcal{V}_{i}}x_{k}\|_2^2 $
		
		\STATE	4: Set $x_{k+1}=x_{k}+A_{\mathcal{V}_{i_{k}}}^{\dagger}\left(b_{\mathcal{V}_{i_{k}}}-A_{\mathcal{V}_{i_{k}}}x_{k}\right)$
		\STATE	5:  \textbf{end for}
	\end{algorithmic}
\end{algorithm}
\begin{definition}\label{def:2.1}
	\cite{ref8} (Row paving) $A (t, \alpha, \beta)$  row paving of an  $m \times n$  matrix $A$ is a partition
	$\mathcal{V}=\{\mathcal{V}_1,\mathcal{V}_2,\ldots,\mathcal{V}_t\}$  of the rows such that
	$$ \alpha\leq\sigma_{min}^2(A_{\mathcal{V}_{i}})\quad and \quad \sigma_{max}^2(A_{\mathcal{V}_{i}})\leq \beta  $$ for each $\mathcal{V}_{i}\in\mathcal{V}$.
	
	Where we denote by $A_{\mathcal{V}_{i}}$ the $|\mathcal{V}_{i}| \times n $ submatrix of $A$. We refer to the number $t$ as the size of the paving, and the numbers $\alpha$ and $\beta$ are called the lower and upper paving bounds.
\end{definition}

For the convergence property of the maximum residual block Kaczmarz(MRBK) method, we can establish the following theorem.

\begin{theorem}\label{thm:2.1}
	Let the linear system \eqref{eq:1.1} be consistent, for a fixed partition $\mathcal{V}=\{\mathcal{V}_1,\mathcal{V}_2,\ldots,\mathcal{V}_t\}$ of $\mathrm[m]$,
	starting from any initial vector $x_{0}\in\mathcal{R}(A^{T})$ , the
	iteration sequence $\{x_{k}\}_{k=0}^{\infty}$  generated by the MRBK method, converges to the unique least-norm solution $x_{\star}=A^{\dagger}b$. Moreover, for any $k\geq0$, we have
	\begin{equation}
		\centering \|x_{1}-x_{\star}\|_{2}^{2}\leq\left(1-\frac{\sigma_{min}^2(A)}{\beta t}\right)\| x_{0}-x_{\star}\|_{2}^{2}
	\end{equation}
	and
	\begin{equation}
		\centering \|x_{k+1}-x_{\star}\|_{2}^{2}\leq\left(1-\frac{\sigma_{min}^2(A)}{\beta(t-1)}\right)^{k}\left(1-\frac{\sigma_{min}^2(A)}{\beta t}\right)\| x_{0}-x_{\star}\|_{2}^{2}
	\end{equation}
\end{theorem}

\begin{proof}
	From the definition of the MRBK method, for a partition $\mathcal{V}=\{\mathcal{V}_1,\mathcal{V}_2,\ldots,\mathcal{V}_t\}$,
	 $k =0, 1, 2, \ldots$  and $i_k \in \{1,2,\ldots,t\}$ , we have $$x_{k+1}-x_{\star}=x_{k}-x_{\star}+A_{\mathcal{V}_{i_{k}}}^{\dagger}\left(b_{\mathcal{V}_{i_{k}}}-A_{\mathcal{V}_{i_{k}}}x_{k}\right)$$
	Since $A_{\mathcal{V}_{i_{k}}}x_{\star}=b_{\mathcal{V}_{i_{k}}}$, it holds that
	$$x_{k+1}-x_{\star}=x_{k}-x_{\star}-A_{\mathcal{V}_{i_{k}}}^{\dagger}A_{\mathcal{V}_{i_{k}}}\left(x_{k}-x_{\star}\right)$$
	Since $A_{\mathcal{V}_{i_{k}}}^{\dagger}A_{\mathcal{V}_{i_{k}}}$ is an orthogonal projector, using the Pythagorean Theorem we have the following relation
	$$\|x_{k+1}-x_{\star}\|_{2}^{2}=\|x_{k}-x_{\star}\|_{2}^{2}-\|A_{\mathcal{V}_{i_{k}}}^{\dagger}A_{\mathcal{V}_{i_{k}}}\left(x_{k}-x_{\star}\right)\|_{2}^{2}$$
	We note that 
	\begin{align}\label{eq:2.3}
		\|A_{\mathcal{V}_{i_{k}}}^{\dagger}A_{\mathcal{V}_{i_{k}}}\left(x_{k}-x_{\star}\right)\|_{2}^{2}&\geq\sigma_{min}^{2}\bigl(A_{\mathcal{V}_{i_{k}}}^{\dagger})\|A_{\mathcal{V}_{i_{k}}}\left(x_{k}-x_{\star}\right)\|_{2}^{2}\nonumber\\
		&=\frac{1}{\sigma_{max}^{2}\bigl(A_{\mathcal{V}_{i_{k}}})}\|A_{\mathcal{V}_{i_{k}}}\left(x_{k}-x_{\star}\right)\|_{2}^{2}\nonumber\\
		&\geq\frac{1}{\beta}\|A_{\mathcal{V}_{i_{k}}}\left(x_{k}-x_{\star}\right)\|_{2}^{2}\nonumber\\
		&=\frac{1}{\beta}\|b_{\mathcal{V}_{i_{k}}}-A_{\mathcal{V}_{i_{k}}}x_{k}\|_{2}^{2}\nonumber\\
		&=\frac{1}{\beta}\max\limits_{1\le i\le t}\|b_{\mathcal{V}_{i}}-A_{\mathcal{V}_{i}}x_{k}\|_{2}^{2}
	\end{align}
	Furthermore, from Method \ref{alg:1}, we have
	$$\begin{aligned}
		b_{\mathcal{V}_{i_{k}}}-A_{\mathcal{V}_{i_{k}}}x_{k+1}&=b_{\mathcal{V}_{i_{k}}}-A_{\mathcal{V}_{i_{k}}}\left(x_{k}+A_{\mathcal{V}_{i_{k}}}^{\dagger}\bigl(b_{\mathcal{V}_{i}}-A_{\mathcal{V}_{i}}x_{k})\right)\\
		&=b_{\mathcal{V}_{i_{k}}}-A_{\mathcal{V}_{i_{k}}}x_{k}-A_{\mathcal{V}_{i_{k}}}A_{\mathcal{V}_{i_{k}}}^{\dagger}\left(b_{\mathcal{V}_{i_{k}}}-A_{\mathcal{V}_{i_{k}}}x_{k}\right)\\
		&=b_{\mathcal{V}_{i_{k}}}-A_{\mathcal{V}_{i_{k}}}x_{k}-A_{\mathcal{V}_{i_{k}}}A_{\mathcal{V}_{i_{k}}}^{\dagger}b_{\mathcal{V}_{i_{k}}}+A_{\mathcal{V}_{i_{k}}}A_{\mathcal{V}_{i_{k}}}^{\dagger}A_{\mathcal{V}_{i_{k}}}x_{k}\\
		&=b_{\mathcal{V}_{i_{k}}}-A_{\mathcal{V}_{i_{k}}}A_{\mathcal{V}_{i_{k}}}^{\dagger}b_{\mathcal{V}_{i_{k}}}-A_{\mathcal{V}_{i_{k}}}x_{k}+A_{\mathcal{V}_{i_{k}}}x_{k}\\
		&=A_{\mathcal{V}_{i_{k}}}x_{\star}-A_{\mathcal{V}_{i_{k}}}A_{\mathcal{V}_{i_{k}}}^{\dagger}A_{\mathcal{V}_{i_{k}}}x_{\star}\\
		&=A_{\mathcal{V}_{i_{k}}}x_{\star}-A_{\mathcal{V}_{i_{k}}}x_{\star}\\
		&=0
	\end{aligned}$$
	Thus we can obtain that for $k =1, 2, \ldots$, that
	\begin{align}\label{eq:2.4}
		\|b-Ax_{k}\|_{2}^{2}&=\sum_{\mathcal{V}_{i_{k}}\in\mathcal{V}\setminus\mathcal{V}_{i_{k-1}}}\|b_{\mathcal{V}_{i_{k}}}-A_{\mathcal{V}_{i_{k}}}x_{k}\|_{2}^{2}\nonumber\\
		&\leq\bigl(t-1)\max\limits_{1\le i\le t}\|b_{\mathcal{V}_{i}}-A_{\mathcal{V}_{i}}x_{k}\|_{2}^{2}
	\end{align}
	
	When $k=0$, we also obtain 
	$$\begin{aligned}
		\|b-Ax_{0}\|_{2}^{2}\leq t\max\limits_{1\le i\le t}\|b_{\mathcal{V}_{i}}-A_{\mathcal{V}_{i}}x_{0}\|_{2}^{2}
	\end{aligned}$$ 
	Hence \begin{align}\label{eq:2.5}
		\max\limits_{1\le i\le t}\|b_{\mathcal{V}_{i}}-A_{\mathcal{V}_{i}}x_{k}\|_{2}^{2}\geq\frac{1}{t-1}\|b-Ax_{k}\|_{2}^{2}
	\end{align}
	Combining \eqref{eq:2.3}, \eqref{eq:2.4} and \eqref{eq:2.5}, we finally have
	$$\begin{aligned}
		\|x_{k+1}-x_{\star}\|_{2}^{2}&\leq\|x_{k}-x_{\star}\|_{2}^{2}-\frac{1}{\beta}\frac{\|b-Ax_{k}\|_{2}^{2}}{t-1}\\
		&=\|x_{k}-x_{\star}\|_{2}^{2}-\frac{1}{\beta}\frac{\|A\bigl(x_{k}-x_{\star})\|_{2}^{2}}{t-1}\\
		&\leq\|x_{k}-x_{\star}\|_{2}^{2}-\frac{\sigma_{min}^{2}(A)}{\beta\bigl(t-1)}\|x_{k}-x_{\star}\|_{2}^{2}\\
		&=\bigl(1-\frac{\sigma_{min}^{2}(A)}{\beta\bigl(t-1)})\|x_{k}-x_{\star}\|_{2}^{2}
	\end{aligned}$$   	
\end{proof} 
From Theorem \ref{thm:2.1}, we find that the convergence factor of the MRBK method is$$\rho_{MRBK}=1-\frac{\sigma_{min}^{2}(A)}{\beta\bigl(t-1)}$$ and from Theorem 3.1 in \cite{ref14} we get that the convergence factor of the GRBK method is $$\rho_{GRBK}=1-\frac{\zeta}{2}\left(\frac{\|A\|_{F}^{2}}{\|A\|_{F}^{2}+\zeta}+1\right)\frac{\sigma_{min}^{2}(A)}{\beta\|A\|_{F}^{2}}$$  where $\zeta= \min\limits_{\mathcal{V}_{i}\in \mathcal{V}}\|A_{\mathcal{V}_{i}}\|_{F}^{2}$.

If the matrix $A$ is a standardized matrix, we obtain from Theorem 2.1 in \cite{ref14} that the convergence factor for the RBK method is $$\rho_{RBK}=1-\frac{\sigma_{min}^{2}(A)}{\beta m}$$
We observe that $\rho_{MRBK}< \rho_{RBK}$ as long as $t-1<m$. In order to compare the convergence factors of the MRBK and GRBK methods, we consider rewriting $\rho_{MRBK}$: $$\rho_{MRBK}=1-\frac{\sigma_{min}^{2}(A)}{\beta\bigl(t-1)}=1-\frac{\|A\|_{F}^{2}}{t-1}\frac{\sigma_{min}^{2}(A)}{\beta\|A\|_{F}^{2}}$$ For the row paving $\{A_{\mathcal{V}_{1}},A_{\mathcal{V}_{2}},\ldots,A_{\mathcal{V}_{t}}\}$ of standardized matrix $A$, we assume that every cardinality of $A_{\mathcal{V}_{i}}$ is equal, i.e. equal to $\frac{m}{t}$. Thus $$\frac{\|A\|_{F}^{2}}{t-1}=\frac{m}{t-1}\quad and \quad \frac{\zeta}{2}\left(\frac{\|A\|_{F}^{2}}{\|A\|_{F}^{2}+\zeta}+1\right)=\frac{m}{2t}\left(\frac{t}{t-1}+1\right)=\frac{1}{2}\left(\frac{m}{t-1}+\frac{m}{t}\right)$$ We note that when the size of the coefficient matrix $A$ is very large, $m$ is usually much larger than $t$, so the convergence factor of the MRBK method is slightly less than or approximately equal to the convergence factor of the GRBK method.

\begin{remark}
	Even if the convergence factors of the MRBK and GRBK methods are almost equal, the convergence rate of the MRBK method is still faster than that of the GRBK method due to the lower computational cost of selecting working row block $A_{\mathcal{V}_{i_{k}}}$. The numerical experiments in Section \ref{sec:4} verify our inference.
\end{remark}

\section{Maximum residual average block Kaczmarz method}
\label{sec:3}

In this section, we consider further improvements to the MRBK method. In step 4 of Method \ref{alg:1}, each update of $x_{k}$ requires the pseudo-inverse of $A_{\mathcal{V}_{i_{k}}}$ to be applied to the vector, which is computationally expensive when the size of matrix $A$ is very large. We develop the Maximum residual average block Kaczmarz (MRABK) method by projecting $x_{k}$ onto each row of $A_{\mathcal{V}_{i_{k}}}$ and averaging them with different extrapolation steps, avoiding the computation of pseudo-inverse and greatly saving the computational cost at each iteration, as shown in Method \ref{alg:2}.
\begin{algorithm}
	\caption{the MRABK Method for solving the linear system}
	\label{alg:2}
	\begin{algorithmic}
		\STATE	\textbf{Input:} $A$, $b$, $\ell$, $\omega \in (0,2)$ and ${x}_{0}$.
		\STATE  \textbf{Output:} ${x}_{\ell}$
		
		\STATE	1: Let $\{\mathcal{V}_1,\mathcal{V}_2,\ldots,\mathcal{V}_t\}$ be a partition of $\mathrm[m]$
		
		\STATE	2: \textbf{for} $k=0,1,\ldots,\ell-1$ \textbf{do}
		\STATE	3: Select $i_{k}=\arg\max\limits_{1\le i\le t}\|b_{\mathcal{V}_{i}}-A_{\mathcal{V}_{i}}x_{k}\|_2^2 $
		\STATE	4:Compute $\alpha_{k}=\omega\frac{\|b_{\mathcal{V}_{i_{k}}}-A_{\mathcal{V}_{i_{k}}}x_{k}\|_{2}^{2}\|A_{\mathcal{V}_{i_{k}}}\|_{F}^{2}}{\|A_{\mathcal{V}_{i_{k}}}^{T}(b_{\mathcal{V}_{i_{k}}}-A_{\mathcal{V}_{i_{k}}}x_{k})\|_{2}^{2}}$
		\STATE	5: Set $x_{k+1}=x_{k}+\alpha_{k}\frac{A_{\mathcal{V}_{i_{k}}}^{T}(b_{\mathcal{V}_{i_{k}}}-A_{\mathcal{V}_{i_{k}}}x_{k})}{\|A_{\mathcal{V}_{i_{k}}}\|_{F}^{2}}$
		\STATE	6:  \textbf{end for}
	\end{algorithmic}
\end{algorithm}

For the convergence property of the maximum residual average block Kaczmarz(MRABK) method, we can establish the following theorem.

\begin{theorem}\label{thm:3.1}
	Let the linear system \eqref{eq:1.1} be consistent, for a fixed partition $\mathcal{V}=\{\mathcal{V}_1,\mathcal{V}_2,\ldots,\mathcal{V}_t\}$ of $\mathrm[m]$ and $\omega \in (0,2)$, starting from any initial vector $x_{0}\in\mathcal{R}(A^{T})$, the
	iteration sequence $\{x_{k}\}_{k=0}^{\infty}$  generated by the MRABK method, converges to the unique least-norm solution $x_{\star}=A^{\dagger}b$. Moreover, for any $k\geq0$, we have
	\begin{equation}
		\|x_{1}-x_{\star}\|_{2}^{2}\leq\left(1-(2\omega-\omega^{2})\frac{\sigma_{min}^2(A)}{\beta t}\right)\| x_{0}-x_{\star}\|_{2}^{2}
	\end{equation}
	and
	\begin{equation}
		\centering \|x_{k+1}-x_{\star}\|_{2}^{2}\leq\left(1-(2\omega-\omega^{2})\frac{\sigma_{min}^{2}(A)}{\beta \bigl(t-1)}\right)^{k}\left(1-(2\omega-\omega^{2})\frac{\sigma_{min}^2(A)}{\beta t}\right)\| x_{0}-x_{\star}\|_{2}^{2}
	\end{equation}
\end{theorem}
\begin{proof}
	From step 5 of Method \ref{alg:2} and $A_{\mathcal{V}_{i_{k}}}x_{\star}=b_{\mathcal{V}_{i_{k}}}$, we can get
	$$\begin{aligned}
		x_{k+1}-x_{\star}&=x_{k}-x_{\star}-\alpha_{k}\frac{A_{\mathcal{V}_{i_{k}}}^{T}(b_{\mathcal{V}_{i_{k}}}-A_{\mathcal{V}_{i_{k}}}x_{k})}{\|A_{\mathcal{V}_{i_{k}}}\|_{F}^{2}}\\
		&=x_{k}-x_{\star}-\alpha_{k}\frac{A_{\mathcal{V}_{i_{k}}}^{T}A_{\mathcal{V}_{i_{k}}}(x_{k}-x_{\star})}{\|A_{\mathcal{V}_{i_{k}}}\|_{F}^{2}}\\
		&=\left(I-\alpha_{k}\frac{A_{\mathcal{V}_{i_{k}}}^{T}A_{\mathcal{V}_{i_{k}}}}{\|A_{\mathcal{V}_{i_{k}}}\|_{F}^{2}}\right)(x_{k}-x_{\star})
	\end{aligned}$$
	Taking the square of the Euclidean norm for both sides of the above equation, it holds that
	$$\begin{aligned}
		\|x_{k+1}-x_{\star}\|_{2}^{2}&=\|\left(I-\alpha_{k}\frac{A_{\mathcal{V}_{i_{k}}}^{T}A_{\mathcal{V}_{i_{k}}}}{\|A_{\mathcal{V}_{i_{k}}}\|_{F}^{2}}\right)(x_{k}-x_{\star})\|_{2}^{2}\\
		&=\|x_{k}-x_{\star}\|_{2}^{2}-2\alpha_{k}\frac{A_{\mathcal{V}_{i_{k}}}(x_{k}-x_{\star})}{\|A_{\mathcal{V}_{i_{k}}}\|_{F}^{2}}+\alpha_{k}^{2}\frac{A_{\mathcal{V}_{i_{k}}}^{T}A_{\mathcal{V}_{i_{k}}}(x_{k}-x_{\star})}{\|A_{\mathcal{V}_{i_{k}}}\|_{F}^{2}}
	\end{aligned}$$
	Substituting $\alpha_{k}$ into this equality, we have
	$$\begin{aligned}
		\|x_{k+1}-x_{\star}\|_{2}^{2}&=\|x_{k}-x_{\star}\|_{2}^{2}-(2\omega-\omega^{2})\frac{\|b_{\mathcal{V}_{i_{k}}}-A_{\mathcal{V}_{i_{k}}}x_{k}\|_{2}^{4}}{\|A_{\mathcal{V}_{i_{k}}}^{T}(b_{\mathcal{V}_{i_{k}}}-A_{\mathcal{V}_{i_{k}}}x_{k})\|_{2}^{2}}\\
		&\leq\|x_{k}-x_{\star}\|_{2}^{2}-(2\omega-\omega^{2})\frac{\|b_{\mathcal{V}_{i_{k}}}-A_{\mathcal{V}_{i_{k}}}x_{k}\|_{2}^{4}}{\sigma_{max}^{2}(A_{\mathcal{V}_{i_{k}}})\|(b_{\mathcal{V}_{i_{k}}}-A_{\mathcal{V}_{i_{k}}}x_{k})\|_{2}^{2}}\\
		&\leq\|x_{k}-x_{\star}\|_{2}^{2}-(2\omega-\omega^{2})\frac{\|b_{\mathcal{V}_{i_{k}}}-A_{\mathcal{V}_{i_{k}}}x_{k}\|_{2}^{2}}{\beta}\\
		&=\|x_{k}-x_{\star}\|_{2}^{2}-(2\omega-\omega^{2})\frac{\max\limits_{1\le i\le t}\|b_{\mathcal{V}_{i}}-A_{\mathcal{V}_{i}}x_{k}\|_{2}^{2}}{\beta}
	\end{aligned}$$
	Where the first of these inequalities is true since that $2\omega-\omega^{2}>0$ for $\omega \in (0,2)$, and $$\|A_{\mathcal{V}_{i_{k}}}^{T}(b_{\mathcal{V}_{i_{k}}}-A_{\mathcal{V}_{i_{k}}}x_{k})\|_{2}^{2}\leq\sigma_{max}^{2}(A_{\mathcal{V}_{i_{k}}})\|(b_{\mathcal{V}_{i_{k}}}-A_{\mathcal{V}_{i_{k}}}x_{k})\|_{2}^{2}$$
	
	From \eqref{eq:2.4}and \eqref{eq:2.5} in proof of Theorem \ref{thm:2.1}, we can obtain that for $k =1, 2, \ldots$ that
	$$\begin{aligned}
		\|x_{k+1}-x_{\star}\|_{2}^{2}&\leq\|x_{k}-x_{\star}\|_{2}^{2}-(2\omega-\omega^{2})\frac{\|b-Ax_{k}\|_{2}^{2}}{\beta \bigl(t-1)}\\
		&\leq\|x_{k}-x_{\star}\|_{2}^{2}-(2\omega-\omega^{2})\frac{\sigma_{min}^{2}(A)\|\bigl(x_{k}-x_{\star})\|_{2}^{2}}{\beta \bigl(t-1)}\\
		&=\left(1-(2\omega-\omega^{2})\frac{\sigma_{min}^{2}(A)}{\beta \bigl(t-1)}\right)\|x_{k}-x_{\star}\|_{2}^{2}
	\end{aligned}$$
	In particular,when $k=0$ , we have 
	$$\|x_{1}-x_{\star}\|_{2}^{2}\leq\left(1-(2\omega-\omega^{2})\frac{\sigma_{min}^2(A)}{\beta t}\right)\| x_{0}-x_{\star}\|_{2}^{2}$$
\end{proof}
\begin{remark}
	When comparing the convergence factors of the the MRBK and MRABK methods, it was observed that when $\omega$ is set to 1, both methods exhibit equivalent convergence factors. Therefore, when applying these two methods to solve consistent linear systems, their iteration steps should be approximately equal. However, due to the utilization of the average block method in the MRABK method to avoid pseudo-inverse calculations during the update of $x_{k}$, the convergence rate of the MRABK method is faster than that of the MRBK method. The numerical experiment results in section \ref{sec:4} verify our conclusions well in the convergence rate.
\end{remark}
\section{Experimental results}
\label{sec:4}

The efficiency of the MRBK and MRABK methods is verified through numerical experiments in this section. We compare these two methods with the GRK, MRK, RBK, GBK and GRBK methods. For the GBK method, we adopt the same value for $\eta$ as mentioned in \cite{ref11}, i.e., $$\eta=\frac{1}{2}+\frac{1}{2}\frac{\|b-Ax_k\|_2^2}{\|A\|_F^2}\left(\max_{1\le i\le m}\left\{\frac{\left|b^{(i)}-A^{(i)}x_k\right|^2}{\left\|A^{(i)}\right\|_2^2}\right\}\right)^{-1}.$$ In each iteration of the RBK, GBK, GRBK and MRBK methods, we utilize CGLS \cite{ref36} instead of calculating the pseudo-inverse of the submatrix of $A$. Additionally, a unified randomized row partition  $\{A_{\mathcal{V}_{1}},A_{\mathcal{V}_{2}},\ldots,A_{\mathcal{V}_{t}}\}$ defined as \eqref{eq:1.2} is implemented in the RBK, GRBK, MRBK and MRABK methods, and for the selection of the number of blocks, \cite{ref10} proves that $\lceil\|A\|_{2}^{2}\rceil $ is a good choice, so we set the number of blocks $t = \lceil\|A\|_{2}^{2}\rceil$ uniformly, where $\lceil\cdot\rceil $ means round up to an integer. In the case of the MRABK method specifically, $\omega$ is set to 1. We measure the performance of the above methods in terms of the number of iteration steps (denoted by``IT'') and the running time in seconds (denoted by ``CPU"), where IT and CPU are the arithmetic average of the number of iteration steps required and CPU time consumed by the corresponding method to repeat 20 times. To demonstrate the effectiveness, we determine the speed-up value of the MRBK method against the MRK (or GRBK) method and the speed-up value of the MRABK method against the MRBK method, which are respectively defined by
$$\text{SU}_{1} = \frac{\text{CPU of MRK}}{\text{CPU of  MRBK}} ,$$ 
$$\text{SU}_{2} = \frac{\text{CPU of GRBK}}{\text{CPU of  MRBK}} ,$$
$$\text{SU}_{3} = \frac{\text{CPU of MRBK}}{\text{CPU of MRABK}} .$$ 

For consistent linear systems \eqref{eq:1.1}, the coefficient matrix is either given by MATLAB function $\mathbf{sprandn}$ or taken from the SuiteSparse Matrix Collection \cite{ref36}.  For the coefficient matrix $A$ being tested, any zero row vectors are removed and $A$ is normalized to the standard matrix.  The right vector is set to $b=Ax_{*}$, where vector $x_{*}$ is the solution vector generated using the MATLAB function $\mathbf{randn}$.  All calculations start with the initial zero vector $x_{0}=0$ and terminate once the relative solution error (RSE) at the current iteration satisfies that $\text{RSE} < 10^{-6}$ or when the number of iteration steps exceeds the maximum of 200000, defined as
$$\text{RSE}=\frac{\|x_{k}-x_{\star}\|_{2}^{2}}{\|x_{\star}\|_{2}^{2}},$$
where the minimum norm solution $x_{\star}$ is obtained using the MATLAB function $\mathbf{lsqminorm}$. All experiments are carried out using MATLAB ((version R2023a) on a personal computer with 1.60 GHz central processing unit (Intel(R) Core(TM) i5-8250U CPU), 8.00 GB memory, and Windows 10 system.

Define the density of a matrix as
$$\text{density} = \frac{\text{number of nonzeros of an m-by-n matrix}}{mn}$$
For the first kind of sparse matrix A, we set its sparse parameter is 0.01, i.e. $\text{A = sprandn (m, n, 0.01)}$, the iteration steps and computing time for the GRK, RBK, GBK, GRBK, MRBK, and MRABK methods are listed in Table \ref{tab:1} and Table \ref{tab:2}, and we also list the speed-up values for several methods.

\begin{table}[!h]
	\centering
	\caption{Numerical results for \(m\)-by-\(n\) random matrices \(A\) with \(m = 6000\) and different \(n\).}
	\label{tab:1}
	\begin{tabular}{lllllll}
		\toprule
		\(m\times n\)&   & 6000$\times$1000 & 6000$\times$1500 & 6000$\times$2000 & 6000$\times$2500 & 6000$\times$3000  \\
		\midrule
		\(\|A\|_2^{2}\) & & 12.29 & 9.13 & 7.64 & 6.63 & 5.86 \\
		\addlinespace
		\multirow{2}{*}{GRK} &IT & 2410.2  & 5203.9  & 10052.8 & 19441.4 & 34776.8 \\
		&CPU & 1.6381 & 3.9657 & 10.4241 & 27.3900 & 56.0059 \\
		\addlinespace
		\multirow{2}{*}{MRK} &IT & 2307.0 & 4929.0  & 9869.0& 19460.0 & 34874.0  \\
		&CPU & 0.5178 & 1.1877 & 3.3316 & 11.6620 & 23.9556\\
		\addlinespace
		\multirow{2}{*}{RBK} &IT & 28.9  & 41.7 & 55.5 & 83.3 & 117.4 \\
		&CPU & 0.0968 & 0.1514 & 0.2877 & 0.4767 & 0.7521\\
		\addlinespace  
		\multirow{2}{*}{GBK} &IT & 41.0  & 74.0 & 114.0 & 200.0 & 325.0 \\
		&CPU & 0.1602 & 0.3303 & 0.7082 & 1.5373 & 2.7144\\
		\addlinespace  
		\multirow{2}{*}{GRBK} &IT & 22.8  & 30.0 & 37.4 & 54.1 & 71.8  \\
		&CPU & 0.1301 & 0.1915 & 0.3297 & 0.5296 & 0.7835\\
		\addlinespace 
		\multirow{2}{*}{MRBK} &IT & \textbf{21.0} & \textbf{29.0} & \textbf{36.0} & \textbf{51.0} & \textbf{68.0} \\
		&CPU & 0.0705 & 0.1100 & 0.2081 & 0.3108 & 0.4529\\
		\addlinespace 
		\multirow{2}{*}{MRABK} &IT & 38.0 & 51.0 & 59.0 & 81.0 & 104.0 \\
		&CPU & \textbf{0.0256} & \textbf{0.0427} & \textbf{0.0859} & \textbf{0.1176} & \textbf{0.1647}\\
		\addlinespace 
		$SU_1$ & & 7.35 & 10.80 & 16.01 & 37.52 & 52.90\\
		\addlinespace
		$SU_2$ & & 1.84 & 1.74 & 1.58 & 1.70 & 1.73\\
		\addlinespace
		$SU_3$ & & 2.75 & 2.58 & 2.42 & 2.64 & 2.75\\  
		\bottomrule
	\end{tabular}
\end{table}

From Table \ref{tab:1}, we can find that when m=6000, n= 1000, 1500, 2000, 2500, 3000, both the MRBK method and MRABK method proposed by us are superior to other methods in terms of computing time. Now we focus on the MRK, GRBK, MRBK and MRABK methods. We observe that the MRBK method, as a block-improved version of the MRK method, is far superior to the latter in terms of iteration steps and computing time. The speed-up value of these two methods ($\text{SU}_{1}$) is at least 7.35 and the maximum is 52.90. In Section \ref{sec:2}, we analyzed the convergence factors of the GRBK and MRBK methods, and concluded that their iteration steps should be very close when the size of $A$ is large. From Table \ref{tab:1}, we can observe that the iteration steps of the two methods are almost the same, and the speed-up value ($\text{SU}_{2}$) is at least 1.58 and at most 1.84. The MRABK method has the shortest computing time among all the above methods, and its speed-up value relative to the MRBK method ($\text{SU}_{3}$) is at least 2.42 and up to 2.75.

From Table \ref{tab:2}, we can find that when n=6000, m= 1000, 1500, 2000, 2500, 3000, the MRBK and MRABK methods are still superior to other methods. Among all the above methods, the MRBK method has the fewest iteration steps, while the MRABK method has the shortest computing time. We note that under these conditions, the speed-up value of the MRBK method to the MRK method ($\text{SU}_{1}$) is at least 48.43 and the maximum is 100.99. Even though the number of iteration steps of the MRBK and GRBK methods are almost the same, the computing time of the MRBK method is better than that of the GRBK method. Compared with the MRBK method, the speed-up value of the MRABK method ($\text{SU}_{3}$) is at least 2.10 and at most 4.11.

\begin{table}[!h]
	\centering
	\caption{Numerical results for \(m\)-by-\(n\) random matrices \(A\) with \(n = 6000\) and different \(m\).}
	\label{tab:2}
	\begin{tabular}{lllllll}
		\toprule
		\(m\times n\)&   & 1000$\times$6000 & 1500$\times$6000 & 2000$\times$6000 & 2500$\times$6000 & 3000$\times$6000  \\
		\midrule
		\(\|A\|_2^{2}\) & & 2.01 & 2.24 & 2.50 & 2.69 & 2.93 \\
		\addlinespace
		\multirow{2}{*}{GRK} &IT & 4115.8  & 8416.0 &15133.5 & 26630.7 & 45296.8 \\
		&CPU & 5.4529 & 13.1930 & 26.1621 & 50.7433 & 95.8928\\
		\addlinespace
		\multirow{2}{*}{MRK} &IT & 4159.0 & 8384.0 & 15018.0 & 27038.0 & 45544.0 \\
		&CPU & 3.1494 & 7.4153 & 14.1870 & 26.9432 & 48.8983\\
		\addlinespace
		\multirow{2}{*}{RBK} &IT & 19.6 & 46.1 & 58.6 & 80.8 & 107.6 \\
		&CPU & 0.1214 & 0.3049 & 0.4108 & 0.6535 & 0.8816	\\
		\addlinespace 
		\multirow{2}{*}{GBK} &IT & 58.0 & 91.0 & 134.0 & 207.0 & 318.0 \\
		&CPU & 0.2652 & 0.5340 & 0.9043 & 1.7528 & 3.1349\\
		\addlinespace 
		\multirow{2}{*}{GRBK} &IT & 10.0 & 21.4 & 28.6 & \textbf{39.6} & 56.8  \\
		&CPU & 0.0775 & 0.1802 & 0.2799 & 0.4661 & 0.6960\\
		\addlinespace 
		\multirow{2}{*}{MRBK} &IT & \textbf{10.0} & \textbf{21.0}  & \textbf{28.0} & 40.0 & \textbf{56.0} \\
		&CPU & 0.0650 & 0.1371 & 0.1988 & 0.3397 & 0.4842\\
		\addlinespace 
		\multirow{2}{*}{MRABK} &IT & 19.0 & 30.0 & 40.0 & 55.0 & 75.0  \\
		&CPU & \textbf{0.0309} & \textbf{0.0449} & \textbf{0.0563} & \textbf{0.0826} & \textbf{0.1264}\\
		\addlinespace 
		$SU_1$ & & 48.43 & 54.10 & 71.35 & 79.31 & 100.99\\
		\addlinespace
		$SU_2$ & & 1.19 & 1.31 & 1.41 & 1.37 & 1.44\\
		\addlinespace
		$SU_3$ & & 2.10 & 3.05 & 3.53 & 4.11 & 3.83\\  
		\bottomrule
	\end{tabular}
\end{table}

For the second type of coefficient matrices selected from the SuiteSparse Matrix Collection \cite{ref35}, which are derived from different applications, such as linear programming problem and combinatorial problem. These matrices have some special structures and properties, such as thin (m\textgreater n) (e.g., Franz9, GL7d12), square (m=n)(e.g., Trefethen\_700), or fat (m\textless n)(e.g., p6000, lp\_80bau3b), we list details about them in Table \ref{tab:3}, including their size, density, and cond ($A$). For these matrices, we implement the GRK, RBK, GBK, GRBK, MRBK and MRABK methods, and each method of iteration steps and computing time are listed in Table \ref{tab:3}.

In Table \ref{tab:3}, we observe that the MRABK method still maintains the shortest computing time. When the coefficient matrix is Franz9 or GL7d12, the GBK method has fewer iteration steps and shorter computing time than MRBK method. When the coefficient matrix is Trefethen\_700, the MRBK method has the smallest number of iteration steps, and when the coefficient matrix is p6000 or  lp\_80bau3b , the number of iteration steps is almost the same as that of the GRBK method. Compared to the MRK method, the MRBK method exhibits a minimum speed-up value of 1.04 and a maximum speed-up value of 28.04. In comparison to the GRBK method, the MRBK method demonstrates a minimum speed-up value of 1.03 and a maximum speed-up value of 1.77. Furthermore, compared to the MRBK method, the MRABK method showcases a minimum speed-up value of 2.37, reaching a maximum speed-up value of 4.05.

\begin{table}[!h]
	\centering
	\caption{Numerical results for matrices \(A\) from the SuiteSparse Matrix Collection.}
	\label{tab:3}
	\begin{tabular}{lllllll}
		\toprule
		name&   & Franz9 & GL7d12 & Trefethen\_700 & p6000 & lp\_80bau3b  \\
		\(m\times n\)&   & 19588$\times$4164 & 8899$\times$1019 & 700$\times$700 & 2095$\times$7967 & 2262$\times$12061  \\
		density & & 0.12\% & 0.41\% & 2.58\% & 0.12\% & 0.09\% \\
		cond($A$) & & 1.46e+16 & Inf & 4.71e+03 & 6.50e+05 & 567.23 \\
		\(\|A\|_2^{2}\) & & 60.80 & 55.75 & 2.54 & 2.35 & 2.82 \\
		\midrule
		\addlinespace
		\multirow{2}{*}{GRK} &IT & 9010.2 & 2604.2 & 1103.2 & 4959.6 & 15856.5 \\
		&CPU & 10.5937 & 1.3168 & 0.1723 & 3.2435 & 12.4952\\
		\addlinespace
		\multirow{2}{*}{MRK} &IT & 8707.5 & 2393.0 & 1093.0 & 4980.0 & 15877.0 \\
		&CPU & 2.9570 & 0.3401 & 0.0573 & 2.1526 & 8.4222\\
		\addlinespace
		\multirow{2}{*}{RBK} &IT & 380.5 & 877.8 & 42.1 & 35.0 & 72.7 \\
		&CPU & 1.8568 & 1.9973 & 0.0597 & 0.1279 & 0.5702\\
		\addlinespace  
		\multirow{2}{*}{GBK} &IT & \textbf{96.2} & \textbf{62.0} & 54.0 & 88.0 & 270.0 \\
		&CPU & 0.5249 & 0.1212 & 0.0638 & 0.3141 & 1.0859\\
		\addlinespace 
		\multirow{2}{*}{GRBK} &IT & 177.0 & 158.9 & 16.9  & \textbf{17.4} & \textbf{34.5} \\
		&CPU & 1.7645 & 0.7016 & 0.0402 & 0.0954 & 0.3300\\
		\addlinespace 
		\multirow{2}{*}{MRBK} &IT & 170.0 & 160.0 & \textbf{12.0} & 18.0  & 36.0 \\
		&CPU & 1.0197 & 0.3274 & 0.0247 & 0.0768 & 0.3206\\
		\addlinespace 
		\multirow{2}{*}{MRABK} &IT & 533.5 & 241.0 & 40.0 & 32.0 & 127.0 \\
		&CPU & \textbf{0.3895} & \textbf{0.0808} & \textbf{0.0070} & \textbf{0.0332} & \textbf{0.0862}\\
		\addlinespace 
		$SU_1$ & & 2.90 & 1.04 & 2.32 & 28.04 & 26.27\\
		\addlinespace
		$SU_2$ & & 1.73 & 1.77 & 1.63 & 1.24 & 1.03\\
		\addlinespace
		$SU_3$ & & 2.62 & 4.05 & 3.52 & 2.31 & 3.72\\  
		\bottomrule
	\end{tabular}
\end{table}

\begin{figure}[!htbp]
	\centering
	\subfigure[]{
		\includegraphics[width=0.45\textwidth]{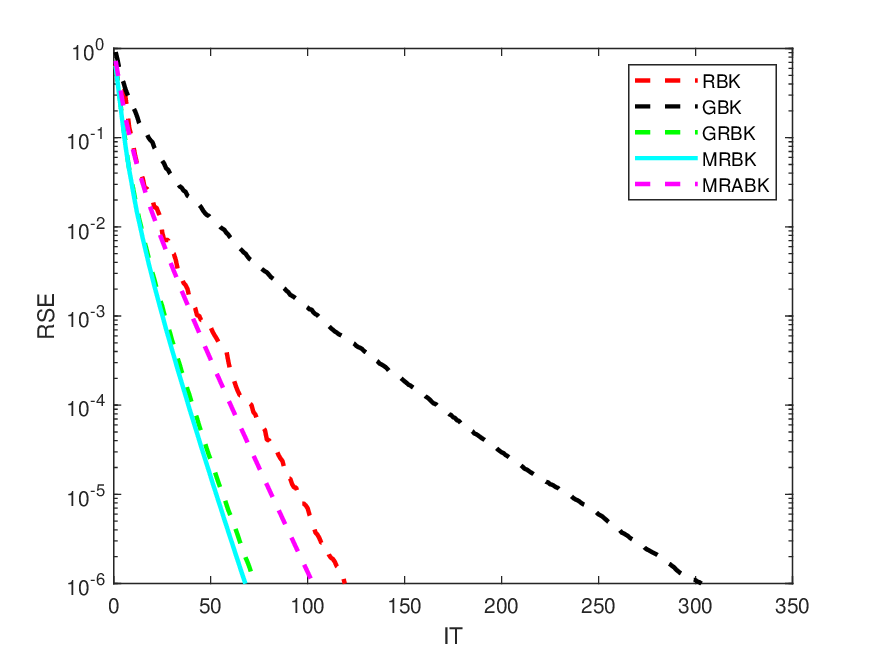}
	}
	\subfigure[]{
		\includegraphics[width=0.45\textwidth]{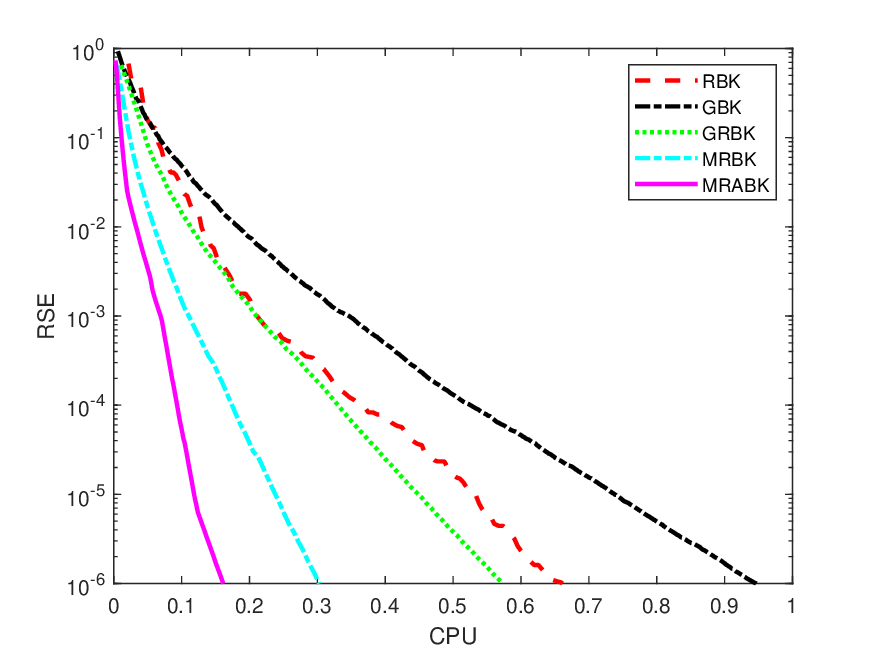}
	}
	\caption{RSE versus IT  (a) and RSE versus CPU (b) for different block Kaczmrarz methods  when the coefficient matrices are the $6000 \times 3000$ matrix in Table \ref{tab:1}.}
	\label{fig:1}
\end{figure}

\begin{figure}[!htbp]
	\centering
	\subfigure[]{
		\includegraphics[width=0.45\textwidth]{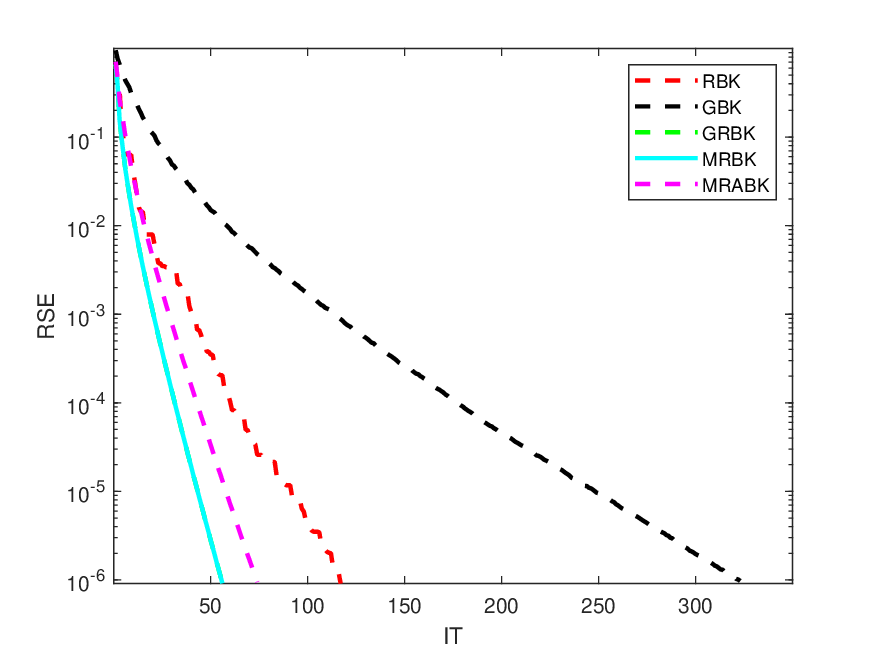}
	}
	\subfigure[]{
		\includegraphics[width=0.45\textwidth]{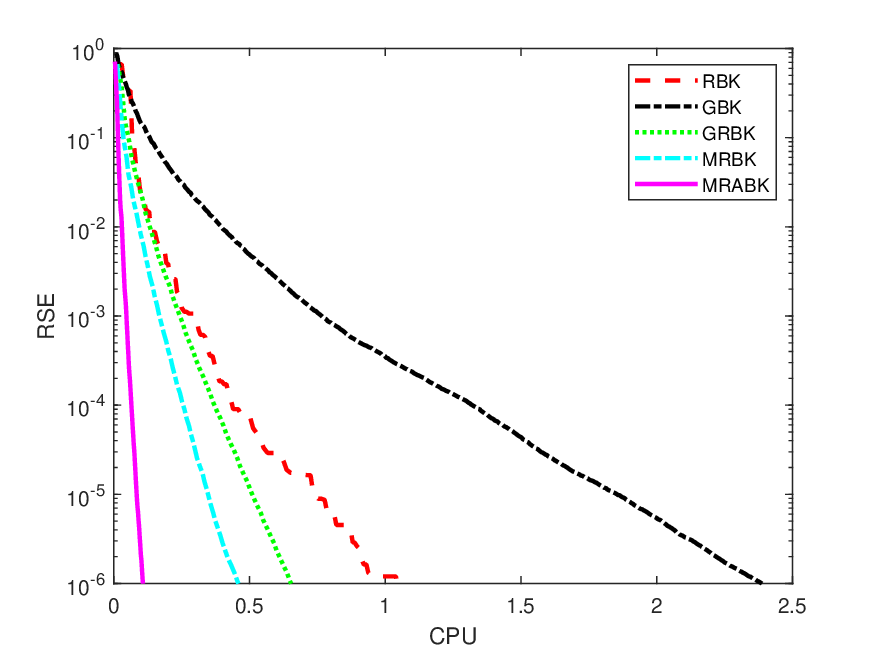}
	}
	\caption{RSE versus IT  (a) and RSE versus CPU (b) for different block Kaczmrarz methods  when the coefficient matrices are the $3000 \times 6000$ matrix in Table \ref{tab:2}. }
	\label{fig:2}
\end{figure}

\begin{figure}[!htbp]
	\centering
	\subfigure[]{
		\includegraphics[width=0.45\textwidth]{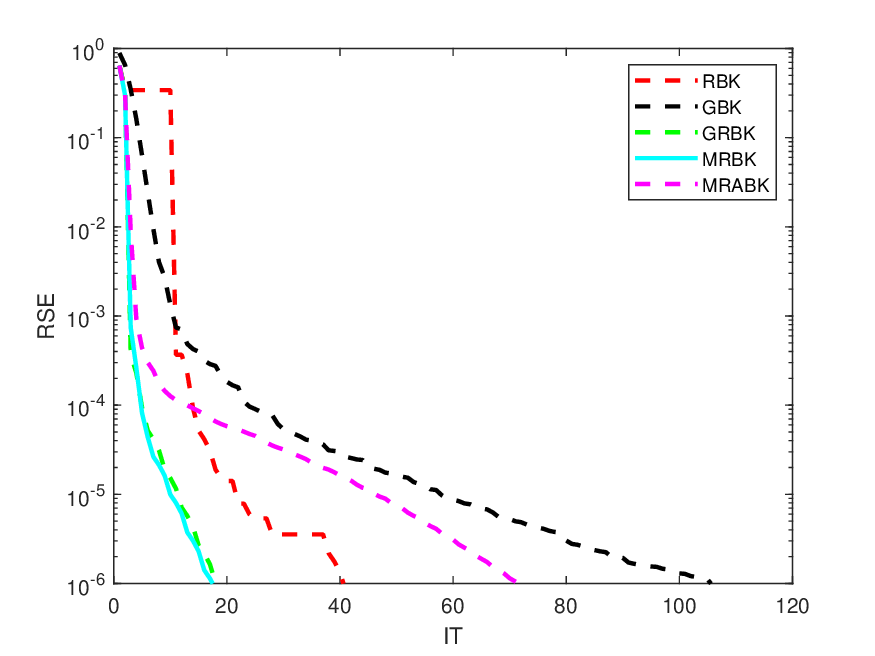}
	}
	\subfigure[]{
		\includegraphics[width=0.45\textwidth]{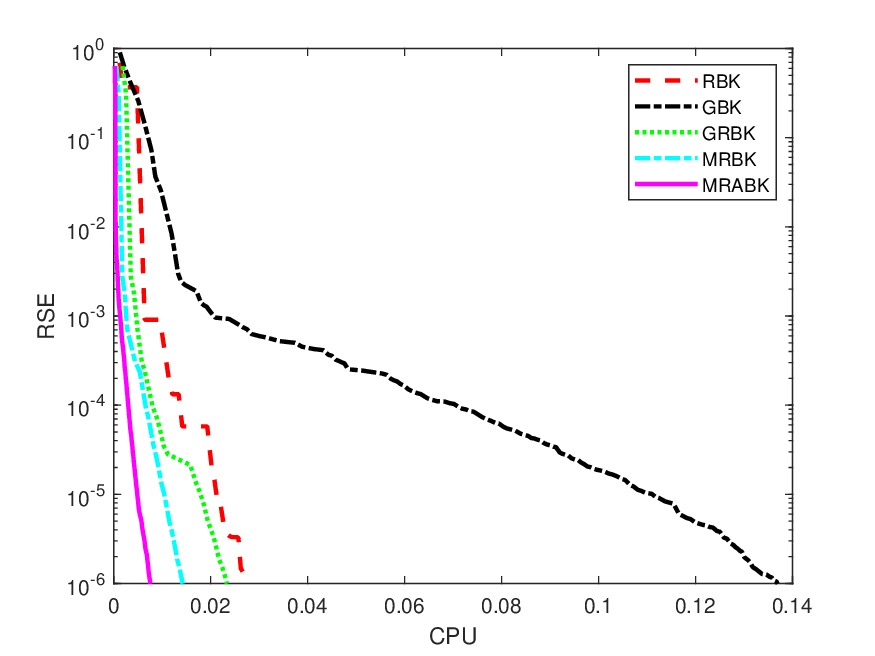}
	}
	\caption{RSE versus IT (a) and RSE versus CPU (b) for different block Kaczmrarz methods on matrix \textbf{Trefethen\_700}. }
	\label{fig:3}
\end{figure}

In order to further analyze the performance of different block Kaczmraz methods, the relationship curves between the relative solution error and the number of iteration steps and the relationship curves between the relative solution error and the computing time for different block Kaczmarz methods were drawn in Figure \ref{fig:1}, Figure \ref{fig:2} and Figure \ref{fig:3}, respectively. Their coefficient matrices are derived from the $6000\times3000$ matrix in Table \ref{tab:1}, the $3000\times6000$ matrix in Table \ref{tab:2}, and the matrix \text{Trefethen\_700} in Table \ref{tab:3}. As we can see from Figure \ref{fig:1}, Figure \ref{fig:2} and Figure \ref{fig:3}, with the increase of the iteration steps, the relative solution errors of the MRBK and MRABK methods both decrease faster than the RBK and GBK methods, and the MRBK method has the fastest decline in relative solution errors among all block Kaczmarz methods. With the increase of the computing time, the relative solution error of the MRABK method decreases the fastest, followed by the MRBK method, both of which are ahead of the RBK, GBK and GRBK methods.

\section{Conclusions}
\label{sec:5}

This paper presents two new Kaczmarz (MRBK and MRABK) methods for solving consistent linear systems. Both methods utilize uniform randomized partition of the rows of matrix $A$ to construct the row subsets $\{A_{\mathcal{V}_{1}},A_{\mathcal{V}_{2}},\ldots,A_{\mathcal{V}_{t}}\}$. In each iteration, the MRBK method updates the solution vector $x$ by projecting $A_{\mathcal{V}_{i_{k}}}$ onto the hyperplane corresponding to $A_{\mathcal{V}_{i_{k}}}$, where $\mathcal{V}_{i_{k}}$ is selected according to  $i_{k}=\arg\max\limits_{1\le i\le t}\|b_{\mathcal{V}_{i}}-A_{\mathcal{V}_{i}}x_{k}\|_2^2 $, ensuring that the block with the largest residual is eliminated first, leading to rapid convergence. Building upon the MRBK method, the MRABK method introduces an adaptive step size, eliminating the need for calculating the pseudo-inverse of the row subset $A_{\mathcal{V}_{i_{k}}}$ of $A$ during the $x$ update process, thereby enhancing the convergence speed. We provide a comprehensive analysis of the convergence theory for these two methods and conduct numerical experiments to validate their effectiveness. Both the theoretical analysis and numerical results demonstrate the superiority of our proposed methods over other Kaczmarz methods, including the GRK, MRK, RBK, GBK, and GRBK methods. In addition, we have realized that some valuable topics deserve further study, such as fully considering the structure and properties of $A$, making more efficient row partition of $A$, and finding a better step size for the MRABK method.

\section*{Declarations}
\textbf{Competing interests} The authors declare no competing interests.




\begin{thebibliography}{60}
	\bibitem[1]{ref1} Kaczmarz S (1937) Angen\"{a}herte aufl\"{o}sung von systemen linearer gleichungen. Bulletin International del’Acad\'{e}mie Polonaise des Sciences et des Lettres S\'{e}rie A 35:355-357
	\bibitem{ref2} Gordon R, Bender R, Herman G (1970) Algebraic reconstruction techniques (ART) for three-dimensional electron microscopy and X-ray photography. J Theor Biol 29:471-481
	\bibitem{ref3} Strohmer T, Vershynin R (2009) A randomized Kaczmarz algorithm with exponential convergence. J Fourier Anal Appl 15(2):262
	\bibitem{ref4}Bai Z-Z, Wu W-T (2018) On greedy randomized Kaczmarz method for solving large sparse linear systems. SIAM J Sci Comput 40(1):A592-A606 
	\bibitem{ref5} Bai Z-Z, Wu W-T (2018) On relaxed greedy randomized Kaczmarz methods for solving large sparse linear systems. Appl Math Lett 83:21-26
	\bibitem{ref6} Zhang J-J (2019) A new greedy Kaczmarz algorithm for the solution of very large linear systems. Appl Math Lett 91:207-212
	\bibitem{ref7} Needell D, Tropp J (2014) Paved with good intentions: analysis of a randomized block Kaczmarz method. Linear Algebra Appl 441(1):199–221
	\bibitem{ref8} Needell D, Zhao R, Zouzias A (2015) Randomized block Kaczmarz method with projection for solving least squares. Linear Algebra Appl 484:322-343
	\bibitem{ref9} Necoara I (2019) Faster randomized block Kaczmarz algorithms. Siam J Matrix Anal A 40:1425-1452
	\bibitem{ref10} Miao C-Q, Wu W-T (2022) On greedy randomized average block Kaczmarz method for solving large linear systems. J Comput Appl Math 413:114372
	\bibitem{ref11} Niu Y-Q, Zheng B (2020) A greedy block Kaczmarz algorithm for solving large-scale linear systems. Appl Math Lett 104:106294
	\bibitem{ref12} Jiang Y, Wu G, Jiang L (2023) A semi-randomized Kaczmarz method with simple random sampling for large-scale linear systems. Adv Comput Math 49:20 
	\bibitem{ref13}Zeng Y, Han D, Su Y, Xie J (2023) Randomized Kaczmarz method with adaptive stepsizes for inconsistent linear systems. Numer Algorithms 94(3):1403-1420
	\bibitem{ref14} Liu Y, Gu C-Q (2021) On greedy randomized block Kaczmarz method for consistent linear systems. Linear Algebra Appl 616:178–200
	\bibitem{ref15} Jiang X-L, Zhang K, Yin J-F (2022) Randomized block Kaczmarz methods with k-means clustering for solving large linear systems. J Comput Appl Math 403:113828
	\bibitem{ref16} Wen L, Yin F, Liao Y, Huang G (2022) A greedy average block Kaczmarz method for the large scaled consistent system of linear equations. AIMS Math 7:6792-6806
	\bibitem{ref17} Li R-R, Liu H (2022) On randomized partial block Kaczmarz method for solving huge linear algebraic systems. Comput Appl Math 41(6):278
	\bibitem{ref18} Gal\'{a}ntai A (2004) Projectors and projection methods. Kluwer Academic Publishers, Boston
	\bibitem{ref19}Knight PA (1993) Error analysis of stationary iteration and associated problems. The University of Manchester, Manchester
	\bibitem{ref20} Brooks MA (2010) A survey of algebraic algorithms in computerized tomography. University of Ontario Institute of Technology, Oshawa
	\bibitem{ref21} Houndfield GN (1973) Computerized transverse axial scanning (tomography): part I, description of system. Brit J Radiol 46:1016–1022
	\bibitem{ref22} Herman GT (2009) Fundamentals of computerized tomography: image reconstruction from projections. Springer, Dordrecht
	\bibitem{ref23} Natterer F (2001) The mathematics of computerized tomography. SIAM, Philadelphia, PA 
	\bibitem{ref24} Herman G, Davidi R (2008) Image reconstruction from a small number of projections. Inverse Probl 24(4):045011
	\bibitem{ref25} Herman G, Meyer L (1993) Algebraic reconstruction techniques can be made computingally efficient (positron emission tomography application). IEEE T Med Imaging 12(3): 600-609
	\bibitem{ref26} Byrne C (2004) A unified treatment of some iterative algorithms in signal processing and image reconstruction. Inverse Probl 20:103-120
	\bibitem{ref27}Lorenz D, Wenger S, Sch\"{o}pfer F, Magnor M (2014) A sparse Kaczmarz solver and a linearized Bregman method for online compressed sensing. In 2014 IEEE international conference on image processing (ICIP):1347-1351
	\bibitem{ref28} Pasqualetti F, Carli R, Bullo F (2012) Distributed estimation via iterative projections with application to power network monitoring. Automatica 48(5):747-758 
	\bibitem{ref29} Elble J, Sahinidis N, Vouzis P (2010) GPU computing with Kaczmarz’s and other iterative algorithms for linear systems. Parallel Comput 36(5-6):215-231
	\bibitem{ref30} Briskman J, Needell D (2015) Block Kaczmarz method with inequalities. J Math Imaging Vis 52:385-396
	\bibitem{ref31} Zhang Y, Li H (2021) Block sampling Kaczmarz–Motzkin methods for consistent linear systems. Calcolo 58(3):39
	\bibitem{ref32} Zhang Y, Li H (2023) Randomized block subsampling Kaczmarz-Motzkin method. Linear Algebra Appl 667:133-150
	\bibitem{ref33} Xiao A-Q, Yin J-F, Zheng, N (2023) On fast greedy block Kaczmarz methods for solving large consistent linear systems. Comput Appl Math 42(3):119
	\bibitem{ref34} Chen J-Q, Huang Z-D (2022) On a fast deterministic block Kaczmarz method for solving large-scale linear systems. Numer Algorithms 89(3):1007-1029
	\bibitem{ref35}Kolodziej S, Aznaveh M, Bullock M, David J, Davis T, Henderson M, Hu Y-F, Sandstrom R (2019) The suitesparse matrix collection website interface. Journal of Open Source Software. 4(35):1244-1248
	\bibitem{ref36}Bj\"{o}rck \r{A} (1996) Numerical methods for least squares problems. SIAM, Philadelphia, PA
	\bibitem{ref37}	Ansorge R (1984) Connections between the Cimmino-method and the Kaczmarz-method for the solution of singular and regular systems of equations. Computing 33:367-375
	\bibitem{ref38} Popa C (2018) Convergence rates for Kaczmarz-type algorithms. Numer Algorithms 79(1):1-17
	\bibitem{ref39} Bai Z-Z, Liu X-G (2013) On the Meany inequality with applications to convergence analysis of several row-action iteration methods. Numer Math 124(2):215-236
	\bibitem{ref40} Dai L, Sch\"{o}n T (2015) On the exponential convergence of the Kaczmarz algorithm. IEEE T Signal Proces 22(10):1571-1574
\end{thebibliography}

\end{document}